\documentclass[12pt]{article}

\usepackage{amsmath, amssymb, amsthm}

\newcommand{\bbZ}{\mathbb{Z}}

\newcommand{\bbR}{\mathbb{R}}
\newcommand{\bbC}{\mathbb{C}}
\newcommand{\bbH}{\mathbb{H}}

\newcommand{\bbF}{\mathbb{F}}
\newcommand{\frg}{\mathfrak{g}}
\newcommand{\frh}{\mathfrak{h}}
\newcommand{\frf}{\mathfrak{f}}
\newcommand{\frb}{\mathfrak{b}}
\newcommand{\frd}{\mathfrak{d}}
\newcommand{\frt}{\mathfrak{t}}
\newcommand{\fru}{\mathfrak{u}}

\newcommand{\frsl}{\mathfrak{sl}}

\newcommand{\frsu}{\mathfrak{su}}

\newcommand{\GL}{GL}
\newcommand{\SL}{SL}
\newcommand{\SO}{SO}
\newcommand{\Sp}{Sp}
\newcommand{\SU}{SU}
\newcommand{\U}{U}
\newcommand{\grT}[1]{T_{#1}}

\newcommand{\ssec}{\Gamma}
\newcommand{\form}{\Omega}
\newcommand{\exter}{\Lambda}
\newcommand{\emrp}{\mathrm{End}}
\newcommand{\lie}{\mathcal{L}}
\newcommand{\dd}{\partial}
\newcommand{\ddb}{\overline\partial}
\newcommand{\ob}{\nabla}
\newcommand{\ii}{\sqrt{-1}}
\newcommand{\EE}{\mathcal{E}}

\newtheorem{defn}{Definition}[section]
\newtheorem{prop}[defn]{Proposition}
\newtheorem{thm}[defn]{Theorem}
\newtheorem{lem}[defn]{Lemma}
{\theoremstyle{remark}
 \newtheorem{rem}[defn]{Remark}

\makeatletter
\@addtoreset{equation}{section}
\makeatother

\title{Holonomy of the Obata connection on $SU(3)$}

\author{Andrey Soldatenkov 
	\thanks{The author is partially supported by
	AG Laboratory HSE, RF government 
	grant, ag. 11.G34.31.0023}}
\date{}

\pagestyle{plain}

\begin{document}

\maketitle

\begin{abstract}
	A hypercomplex structure on a smooth manifold is a triple of integrable
	almost complex structures satisfying quaternionic relations. The Obata
	connection is the unique torsion-free connection that preserves each of the
	complex structures. The holonomy group of the Obata connection
	is contained in $\GL(n, \bbH)$. There is a well-known construction of
	hypercomplex structures on Lie groups due to Joyce. In this paper
	we show that the holonomy of the Obata connection on $\SU(3)$ coincides
	with $\GL(2, \bbH)$.
\end{abstract}

\tableofcontents

\section{Introduction}

Consider a smooth manifold equipped with a triple of almost
complex structures satisfying quaternionic relations.
The manifold is called hypercomplex if these almost complex
structures are integrable. Hypercomplex manifolds were 
defined by Boyer \cite{Bo} and they
were much studied since then. There exist many examples of such manifolds
including hyperk\"ahler manifolds, nilmanifolds, Lie groups
with hypercomplex structures and others. Boyer also 
classified compact hypercomplex manifolds of real dimension four.
Homogeneous hypercomplex structures on Lie groups appeared in the context of 
string theory (see \cite{SSTV}) and then in the work of Joyce \cite{J}.

Each hypercomplex manifold is endowed with a torsion-free connection
preserving all the complex structures which is called the Obata connection.
The holonomy group of this connection is an important characteristic of the
hypercomplex structure. Since the Obata connection preserves the
quaternionic structure, its holonomy is contained in $\GL(n,\bbH)$ which is
one of the groups in the list of possible irreducible holonomies. 

The classification of irreducible holonomy groups of torsion-free
connections has a long history. For locally symmetric connections
the problem essentially reduces to the classification of symmetric
spaces which was known since \'Elie Cartan (see e.g. \cite{Bes}). 
For connections that are not locally symmetric a major breakthrough
was made in 1955 by Berger. He obtained a list
of irreducible metric holonomies (i.e. holonomies of the connections
that preserve some non-degenerate symmetric bilinear form) and a part
of the list of non-metric ones. The classification was completed
in 1999 by Merkulov and Schwachh\"ofer \cite{MS} thus providing
a full list of all possible irreducible holonomy groups.

The subgroups of $\GL(n,\bbH)$ which appear in the list of irreducible
ho\-lo\-no\-mies are $\Sp(n)$ and $\SL(n,\bbH)$. For both of these subgroups,
there exist examples of manifold with holonomy contained in it.
These are hyperk\"ahler manifolds for $\Sp(n)$ and, for example, nilmanifolds
for $\SL(n,\bbH)$ (see e.g. \cite{BDV}).
The group $\GL(n,\bbH)$ appears as a possible local holonomy group (see \cite{MS}),
but it was apparently unknown if it could occur as a holonomy of a compact
hypercomplex manifold.
The purpose of the present paper is to prove that the holonomy of
the Obata connection on $\SU(3)$ is $\GL(2,\bbH)$, thus providing
the first compact example.

In Section~\ref{2} we recall the definition of the hypercomplex
structure and obtain some useful properties of the Obata connection.
In Section~\ref{3} we review the construction of the hypercomplex structures on
Lie groups. In Section~\ref{4} we study the Obata connection on $\SU(3)$ and
prove the main theorem.

\medskip
\noindent
{\bf Acknowledgements.} I would like to express my gratitude to
Misha Verbitsky for suggesting the topic of this work to me, 
for fruitful discussions and constant encouragement.

\section{Hypercomplex manifolds and the Obata connection}\label{2}

In this section we recall the definition of a hypercomplex manifold
and establish some useful properties of the Obata connection.

\subsection{Hypercomplex structures}

Let $M$ be a smooth manifold. Recall that an {\it almost complex structure}
on $M$ is an endomorphism $I\colon TM\to TM$ satisfying $I^2 = -Id$.
The {\it Nijenhuis tensor} for $I$ is given by
\begin{equation}\label{eqn_Nijen}
N_I(X,Y) = [X,Y] + I[IX,Y] + I[X,IY] - [IX,IY].
\end{equation}

If the Nijenhuis tensor vanishes, the almost complex structure is called
integrable. It is a well-known result of Newlander and Nirenberg that 
every integrable almost complex structure arises from a complex analytic
structure on $M$.

\begin{defn}\label{defn_Hypercompl}
A hypercomplex structure on a smooth manifold $M$ is a triple of
integrable almost complex structures $I$, $J$, $K$ satisfying $$IJ = -JI = K.$$
\end{defn}

Note that a hypercomplex structure induces a natural action of the
quaternion algebra $\bbH$ on the tangent bundle of $M$. Thus, every
hypercomplex manifold is equipped with a two-dimensional sphere of
complex structures corresponding to imaginary quaternions of unit length.

\subsection{The Obata connection}

Let $(M,I,J,K)$ be a hypercomplex manifold. It has been shown by Obata
\cite{Ob} that $M$ admits a unique torsion-free connection $\ob$ that
preserves the hypercomplex structure, i.e.
$$
\ob I = \ob J = \ob K = 0.
$$
This connection is called {\it the Obata connection}.

Consider the decomposition of the complexified
tangent bundle of $M$ with respect to $I$:
$$
T_{\bbC}M = TM \otimes_{\bbR} \bbC = T^{1,0}_IM \oplus T^{0,1}_IM,
$$
where $T^{1,0}_IM = \{X \in T_{\bbC}M\colon IX = \ii X\}$, \
$T^{0,1}_IM = \{X \in T_{\bbC}M\colon IX = -\ii X\}$.

Since the complex structure $I$ anticommutes with $J$, the latter interchanges
the eigenspaces of $I$:
$$
J\colon T^{1,0}_IM \to T^{0,1}_IM, \qquad J\colon T^{0,1}_IM \to T^{1,0}_IM.
$$

Recall that the bundle $T^{1,0}_IM$ can be endowed with a
holomorphic structure given by an operator
$$
\ddb\colon \ssec(T^{1,0}_IM) \to \form^{0,1}_IM \otimes \ssec(T^{1,0}_IM),
$$
where $\form^{0,1}_IM$ is a space of $(0,1)$-forms with respect to $I$. 
Similarly, $T^{0,1}_IM$ can be endowed with an antiholomorphic structure
$$
\dd\colon \ssec(T^{0,1}_IM) \to \form^{1,0}_IM \otimes \ssec(T^{0,1}_IM).
$$

We will identify the complex bundle $(TM, I)$ with $T^{1,0}_IM$ via the
isomorphism
\begin{equation}\label{eqn_Iso}
X \mapsto \frac12 (X - \ii IX).
\end{equation}
Since $\ob$ preserves $I$, this isomorphism enables us to view 
the Obata connection as a connection on $T^{1,0}_IM$.
Considered from this perspective, $\ob$ admits an especially
simple description.

\begin{prop}\label{prop_ob}
Let $(M,I,J,K)$ be a hypercomplex manifold, $\dim_{\bbR}M = 4n$.
\begin{enumerate}
\item 
	The Obata connection  
	$\ob\colon \ssec(T^{1,0}_IM) \to \form_\bbC M \otimes \ssec(T^{1,0}_IM)$
	is given by
	\begin{equation}\label{eqn_Obata}
	\ob = \ddb - J\, \dd\, J.
	\end{equation}
\item
	The curvature of the Obata connection is an $\SU(2)$-invariant
	2-form with coefficients in $\emrp_{\bbH}(TM)$, where $\SU(2)$ is
	identified with the group of unit quaternions:
	$$
	R(IX,IY)Z = R(JX,JY)Z = R(KX,KY)Z = R(X,Y)Z.
	$$
\end{enumerate}
\end{prop}
\begin{proof}
It is clear that the formula (\ref{eqn_Obata}) defines a complex connection
on $T^{1,0}_IM$. Note that under identification (\ref{eqn_Iso}) the
endomorphism $J$ of the real tangent bundle maps to complex-antilinear
operator $A\colon T^{1,0}_IM\to T^{1,0}_IM$, $AX = J\overline{X}$.
A short calculation shows that $\ob$ preserves $A$:
\begin{eqnarray}
	(\ob A)X &=& \ob (AX) - A \ob X \nonumber \\ 
	&=& \ddb J\overline{X} - J\,\dd\,J^2\overline{X} - 
	 J\left(\overline{\ddb X}\right) + 	 J\left(\overline{J\,\dd\,JX}\right) \nonumber \\
	&=& \ddb J\overline{X} + J\,\dd \overline{X} - J\,\dd \overline{X} -
	 \ddb (\overline{JX}) = 0. \nonumber
\end{eqnarray}
This proves that $\ob$ preserves the hypercomplex structure.

It remains to check that the corresponding connection on $TM$ is torsion-free.
Let $e_i = \frac{1}{2}(\xi_i - \ii I\xi_i), i = 1,\ldots,2n$ be a local
holomorphic basis of $T^{1,0}_IM$, where $\xi_i$ are pairwise commuting real
vector fields. We have to show that $\ob_{\xi_i}{\xi_j} = \ob_{\xi_j}{\xi_i}$
and $\ob_{I\xi_i}{\xi_j} = \ob_{\xi_j}{I \xi_i}$ for all $i,j=1,\ldots,2n$.
Note that in view of the isomorphism (\ref{eqn_Iso}) 
$$
\ob_{\xi_i}\xi_j = \ob_{\xi_i}e_j = \ob^{1,0}_{e_i}e_j,
$$
because $\ob^{0,1}e_j = \ddb e_j = 0$. So it suffices to show that
$\ob^{1,0}_{e_i}{e_j} = \ob^{1,0}_{e_j}{e_i}$, which is equivalent
to $\dd_{e_i}{Je_j} = \dd_{e_j}{Je_i}$ according to (\ref{eqn_Obata}).

Consider the vector fields 
$$
f_i = e_i - \ii Je_i \in T^{1,0}_JM.
$$
Since the almost complex structure $J$ is integrable we have
$[f_i, f_j] \in T^{1,0}_JM$. We claim that $[f_i,f_j]$ is also contained in
$T^{0,1}_IM$. Indeed, 
$$
[f_i,f_j] = -[Je_i,Je_j]-\ii([Je_i,e_j]+[e_i,Je_j]).
$$
But since $Je_i \in T^{0,1}_IM$ we have $[Je_i,Je_j] \in T^{0,1}_IM$; \
moreover, because $e_i$ are holomorphic
$[Je_i,e_j] = -\dd_{e_j}Je_i \in T^{0,1}_IM$ and
$[e_i,Je_j] = \dd_{e_i}Je_j \in T^{0,1}_IM$.
So we have proved that $[f_i,f_j] \in T^{0,1}_IM \cap T^{1,0}_JM$. But
the operators $I$ and $J$ anticommute and the intersection of their
eigenspaces is trivial. We conclude that 
$[e_i - \ii Je_i,e_j - \ii Je_j] = 0$. Analogously,
$[e_i + \ii Je_i,e_j + \ii Je_j] = 0$, and it follows from these two
equalities that $\dd_{e_i}{Je_j} - \dd_{e_j}{Je_i} = [e_i,Je_j]-[e_j,Je_i] = 0$.
This completes the proof of the first part.

To prove the second part, note that according to (\ref{eqn_Obata})
$\ob^{1,0}= -J\,\dd\,J$, $\ob^{0,1} = \ddb$ and 
since $\ddb^2 = 0$, $\dd^2 = 0$, $J^2 = -Id$ we have
$\left(\ob^{0,1}\right)^2 = 0$, $\left(\ob^{1,0}\right)^2 = 0$.
The standard argument (which works for the Chern connection, for example)
shows that the curvature $R$ is contained in
$\exter^{1,1}_IM \otimes \emrp(T^{1,0}_IM)$, and this implies $R(IX,IY)Z = R(X,Y)Z$.
Next, note that the complex structure $I$ has been chosen
arbitrarily from the whole 2-dimensional sphere of complex structures
on $M$, thus if we replace $I$ with $J$ and $K$ the analogous reasoning
shows that $R(JX,JY)Z = R(KX,KY)Z = R(X,Y)Z$.
Finally, for every $X$ and $Y$ the endomorphism $R(X,Y)$ is $\bbH$-linear
since the Obata connection preserves the hypercomplex structure. 
\end{proof}

In order to study the hypercomplex structures on Lie groups, it will be
convenient to express the Obata connection in terms of the commutator
of real vector fields. We are going to use the following well-known
formula for the $\ddb$-operator (see \cite{Ga}, where this operator
appears in a similar fashion).

\begin{prop}\label{prop_ddb}
Let $M$ be a smooth manifold and $I$ a complex structure on it.
Considering $(TM,I)$ as a holomorphic bundle (using the isomorphism
(\ref{eqn_Iso})), we can write the corresponding $\ddb$-operator as
\begin{equation}\label{eqn_ddb}
\ddb_X Y = \frac12([X,Y]+I[IX,Y]).
\end{equation}
\end{prop}
\begin{proof}
It is clear that (\ref{eqn_ddb}) is $\bbR$-linear in both $X$ and $Y$.
Moreover, since the Nijenhuis tensor of $I$ (\ref{eqn_Nijen})
vanishes we see that $\ddb_X(IY) = I\ddb_X Y$, i.e (\ref{eqn_ddb}) is
$\bbC$-linear in $Y$. Next, observe that it satisfies the Leibniz rule:
\begin{eqnarray}
	\ddb_X(fY) &=& \frac{1}{2}([X,fY]+I[IX,fY]) \nonumber \\
	&=& \frac{1}{2}(f([X,Y]+I[IX,Y]) + (\lie_Xf) Y + (\lie_{IX}f)IY) \nonumber\\
	&=& f\ddb_X Y + \frac{1}{2}(\lie_Xf + \ii\lie_{IX}f) Y = f\ddb_X Y + 
		(\ddb_X f) Y, \nonumber
\end{eqnarray}
and that it is $C^\infty(M)$-linear in $X$:
$$
	\ddb_{fX}Y = \frac{1}{2}([fX,Y]+I[fIX,Y]) = f\ddb_XY - 
	\frac{1}{2}(\lie_Yf)(X + I^2X) = f\ddb_XY.
$$

Next we have to show that (\ref{eqn_ddb}) vanishes when $Y$ is holomorphic.
But it is known that $Y$ is a holomorphic section of $(TM,I)$ if and only
if $\lie_Y I = 0$. Now, $(\lie_Y I)(IX) = [Y,I^2X]-I[Y,IX] = 2\ddb_X Y = 0$.

Since the properties that we have checked above uniquely determine the
$\ddb$-operator of a holomorphic vector bundle, this completes the proof. 
\end{proof}

It follows from Propositions \ref{prop_ob} and \ref{prop_ddb} that the
Obata connection on a hypercomplex manifold $(M,I,J,K)$ can be written in
the following form:
\begin{equation}\label{eqn_Obata2}
\ob_XY = \frac{1}{2}\big([X,Y]+I[IX,Y]-J[X,JY]+K[IX,JY]\big).
\end{equation}

\section{Hypercomplex structures on Lie groups}\label{3}

In this section we review the construction of homogeneous hypercomplex
structures on compact Lie groups following Joyce \cite{J}. Let
$G$ be a compact semisimple Lie group and $\frg$ its Lie algebra.
Let $\frt \subset \frg$ be a maximal torus.

The first step in constructing the hypercomplex structure is to obtain the
following decomposition of $\frg$ (cf. \cite{J}, Lemma 4.1):
$$
	\frg = \frb \oplus \bigoplus_{k=1}^n \frd_k \oplus \bigoplus_{k=1}^n \frf_k,
$$
where $\frb$ is an abelian subalgebra, $\frd_k$ are subalgebras isomorphic
to $\mathfrak{su}(2)$ and $\frf_k$ are subspaces with the following properties:
\begin{enumerate}
\item
	$[\frd_k, \frb] = 0$ and $\frt \subset \frb \oplus \bigoplus_{k=1}^n \frd_k$;
\item 
	$[\frd_k, \frf_j]=0$ for $j > k$;
\item
	$[\frd_k, \frf_k] \subset \frf_k$ and this Lie bracket action
	of $\frd_k$ on $\frf_k$ is isomorphic
	to the direct sum of some number of copies of
	$\mathfrak{su}(2)$-action on $\bbC^2$ by matrix multiplication from the left.
\end{enumerate}

Note that $\frd_k \oplus \fru(1)\simeq \frsu(2)\oplus \fru(1)$ can be identified
with the quaternion algebra $\bbH$. Since the subalgebra $\frb$ is isomorphic
to a direct sum of $\fru(1)$'s we can (after possibly adding some
extra copies of $\fru(1)$, i.e. multiplying $G$ by some number of $S^1$) identify
$\frb \oplus \bigoplus_{k=1}^n \frd_k$ with $\bbH^m$ for some $m$. Denote by
$\mathcal{I}_k, \mathcal{J}_k, \mathcal{K}_k$ the elements of $\frd_k$ corresponding
to the standard imaginary quaternions under the identification 
$\frd_k \oplus \fru(1)\simeq \bbH$. We define a triple of complex structures
$I, J, K \in \emrp(\frg)$ as follows: the action of $I, J, K$ on 
$\frb \oplus \bigoplus_{k=1}^n \frd_k \simeq \bbH^m$ is multiplication
by the corresponding imaginary quaternion from the left and the action on $\frf_k$ is
given by
$$
IX = [\mathcal{I}_k, X], \qquad JX = [\mathcal{J}_k, X], \qquad KX = [\mathcal{K}_k, X]
$$
for $X \in \frf_k$. The endomorphisms $I$, $J$, $K$ define three left-invariant
almost-complex structures on $G$. One can check (\cite{J}, Lemma 4.3) that they
are integrable and satisfy the quaternionic relations thus giving a hypercomplex
structure on $G$.

We are interested in the case when $G = \SU(3)$. The Lie algebra 
$\frg$ is the algebra of $3\times 3$ skew-Hermitian trace-free 
matrices. Such a matrix can be represented in the form
\begin{equation}\label{eqn_matr}
\begin{pmatrix}
D                         & f\\
-\overline{f}\mathstrut^t & b
\end{pmatrix}
\end{equation}
where $D\in \fru(2)$, $f\in \bbC^2$ is a column-vector and $b\in \bbC$
with $\mathrm{tr}(D) + b =0$. The decomposition of $\frg$ described above
takes form $\frg = \frb\oplus\frd\oplus\frf$ where $\frd$ consists of
matrices with zero $f$ and $b$, $\frf$ --- of matrices with 
zero $D$ and $b$ and $\frb$ consists
of diagonal matrices commuting with $\frd$. Note that the adjoint action
of $\frb$ preserves $\frf$ and $[\frf,\frf]\subset \frb\oplus \frd$
thus we obtain $\bbZ/2\bbZ$-grading: $\frg = \frg_0 \oplus \frg_1$ with 
$\frg_0 = \frb\oplus \frd$ and $\frg_1 = \frf$.

We can also mention that it is possible to choose the identification
$\frb\oplus \frd \simeq \bbH$ and thus the corresponding hypercomplex structure
in such a way that the Killing form will be quaternionic Hermitian.
This turns $G$ into an HKT-manifold \cite{GP}.

\section{Holonomy of the Obata connection}\label{4}

\subsection{The Euler vector field}

Consider the Lie group $G = \SU(3)$ with the hypercomplex
structure described above. The Lie algebra of $G$
is $\bbZ/2\bbZ$-graded: $\frg = \frg_0 \oplus \frg_1$,
where $\frg_0 \simeq \mathfrak{su}(2)\oplus\mathfrak{u}(1)$
will be identified with the algebra of quaternions $\bbH$,
and $\frg_1$ is a $\frg_0$-module with the action of $\bbH$
obtained from the adjoint action of $\frg_0$ as described in
the previous section.

We will identify the elements of $\frg$ and left-invariant
vector fields on $G$. Denote by $\EE$ the element of $\frg_0$
(and the vector field) corresponding to $-1\in \bbH$ under
the isomorphism $\frg_0 \simeq \bbH$. We will call $\EE$ the
Euler vector field. Choose also some non-zero element $W\in \frg_1$.
Then $\langle\EE, W\rangle$ form an $\bbH$-basis in $\frg$.
Recall that the action of $\bbH$ on $\frg_1$ is given by
$$
IW = [W,I\EE],\qquad JW=[W,J\EE],\qquad KW=[W,K\EE].
$$

\begin{rem}\label{rem_hopf_surf}
Note that the subgroup $G_0$ corresponding to $\frg_0$ is
isomorphic to $\SU(2)\times \U(1)$ and it is a hypercomplex 
submanifold of $G$. If we identify $\frg_0$ with the quaternion
algebra $\bbH$ then the hypercomplex structure is given
by left quaternionic multiplication. It follows from (\ref{eqn_Obata2})
that the Obata connection in this case is given by 
$$
\ob_X Y = -Y\cdot X
$$
for any $X, Y \in \frg_0$, where $\cdot$ is multiplication in $\bbH\simeq \frg_0$.
It is easy to check that the Obata connection on $G_0$ is flat.
The group $G_0 \simeq \SU(2)\times \U(1)$ is diffeomorphic to a Hopf
manifold $(\bbR^4\backslash\{0\})/\Gamma$, where $\Gamma$ is an infinite cyclic
group generated by the homothety $z\mapsto \lambda z$ for some 
$\lambda\in \bbR_{>0}$. The vector field on $G_0$ corresponding to $\EE \in \frg_0$
lifts to the ordinary Euler vector field on $\bbR^4\backslash\{0\}$ which
generates the flow of homotheties. It is remarkable that the Euler 
vector field $\EE$ on $\SU(3)$ retains some useful properties,
as we show in the following proposition.
It should be mentioned that the vector field $\EE$ appeared in \cite{PPS}, but
the notation in that paper slightly differs from ours.
\end{rem}

\begin{prop}\label{prop_Euler}
The vector field $\EE$ possesses the following properties:
\begin{enumerate}
\item 
	$\EE$ is holomorphic with respect to $I$, $J$, $K$;
\item
	$\ob \EE = Id$, where $Id$ is understood as a section of 
	$\exter^1G\otimes TG \simeq \emrp(TG)$;
\item
	$\ob^2 \EE = 0$;
\item
	If we denote by $h$ the Killing form on $\frg$, then 
	$$
	\ob_\EE h = -2h,\qquad \ob_{I\EE} h = \ob_{J\EE} h = \ob_{K\EE} h = 0.
	$$
\end{enumerate}
\end{prop}
\begin{proof}
\begin{enumerate}
\item
	We have $(\lie_\EE I)X = [\EE,IX]-I[\EE,X]$ which obviously equals zero
	when $X\in \frg_0$ since $\EE$ lies in the center of $\frg_0$. If $X\in \frg_1$
	then $IX = [X,I\EE]$ and $[\EE,[X,I\EE]] = [[\EE,X],I\EE] = I[\EE,X]$, so
	again $(\lie_\EE I)X = 0$. The same argument applies to $J$ and $K$.
\item
	For $X \in \frg_0$ we have $\ob_X\EE = - X\cdot \EE = X$ (see Remark 
	\ref{rem_hopf_surf}).
	
	Now suppose that $X \in \frg_1$.
	It follows from 1 that $\ddb \EE = 0$ and in view of (\ref{eqn_ddb})
	and (\ref{eqn_Obata2}) 
	$$
	\ob_X\EE = \frac{1}{2}(-J[X,J\EE]+K[IX,J\EE]) = \frac{1}{2}(-J^2X+KJIX) = X.
	$$
\item
	Immediately follows from 2.
\item
	A straightforward computation using the bi-invariance of the Killing form:
	\begin{eqnarray}
		(\ob_\EE h)(X,Y) &=& -h(\ob_\EE X, Y) - h(X, \ob_\EE Y) \nonumber\\
		&=& - h(\ob_X\EE + [\EE,X],Y) - h(X,\ob_Y\EE + [\EE,Y]) \nonumber\\
		&=& -2h(X,Y). \nonumber
	\end{eqnarray}
	The last three equalities are obtained analogously using the fact that
	$h$ is quaternionic Hermitian.
\end{enumerate}
\end{proof}

\begin{rem}\label{rem_euler}
Note that if $M$ is a compact manifold with a torsion-free connection $\nabla$
then the existence of a vector field $\EE$ with $\nabla \EE = Id$
has some strong implications for $\nabla$. Namely, observe that for any
vector field X we have $\nabla_\EE X = X + \lie_\EE X$ and
$\nabla_\EE \alpha = -\alpha + \lie_\EE \alpha$ for any 1-form
$\alpha$. Next, take a tensor field of type $(k,m)$: 
$T\in\ssec\left((TM)^{\otimes k}\otimes(T^*M)^{\otimes m}\right)$.
Representing $T$ locally as a sum of the elements of the form 
$X_1\otimes\ldots\otimes X_k\otimes\alpha_1\otimes\ldots\otimes\alpha_m$,
we obtain $\nabla_\EE T = (k-m)T+\lie_\EE T$. Suppose that $\nabla$ preserves
$T$; then $\lie_\EE T = (m-k)T$. If $T$ is non-zero at some point, take an
integral curve of $\EE$ through this point and observe that unless $m=k$
the norm (with respect to an arbitrary metric) of $T$ restricted to this integral
curve will tend to infinity which is impossible for compact $M$.
This means that $\nabla$ 
can preserve tensor fields only of type $(k,k)$, as opposed to, say, Levi-Civita
connection. Note also that the vector field $\EE$ is always unique when it exists,
for if $\nabla \EE' = Id$ then $\nabla$ preserves $\EE - \EE'$ and therefore
 $\EE - \EE' = 0$.
\end{rem}

\subsection{Computation of the holonomy}

We will need the following technical lemma.

\begin{lem}\label{lem_1}
Denote by $R$ the curvature of the Obata connection.
\begin{enumerate}
\item
	$R(X,IX)X + JR(X,KX)X - KR(X,JX)X = 0$ for all $X$;
\item
	Suppose that $\mathcal{Z}$ is a vector field such that
	$R(X,Y)\mathcal{Z} = 0$ for any vector fields $X$ and $Y$.
	Then $R(\mathcal{Z},X)X = 0$ for all $X$.
\end{enumerate}
\end{lem}
\begin{proof}
We will use the first Bianchi identity (which is true for any
torsion-free connection) and the fact that the curvature
of the Obata connection is an $\SU(2)$-invariant 2-form with
coefficients in $\bbH$-linear endomorphisms by Proposition \ref{prop_ob}.
We have:
\begin{eqnarray}
	R(X,IY)Z &=& R(Z,IY)X + R(X,Z)IY \nonumber \\
	&=& R(Z,IY)X + IR(Y,Z)X + IR(X,Y)Z, \nonumber
\end{eqnarray}
where the second equality follows from $\bbH$-linearity of $R(X,Z)$ and
the first Bianchi identity. Similarly
$$
	R(X,IY)Z = R(Y,IX)Z = R(Z,IX)Y + IR(Y,Z)X,
$$
and we obtain the following identity for any vector fields $X$, $Y$, $Z$:
$$
R(Z,IX)Y = R(Z,IY)X + IR(X,Y)Z.
$$
Substituting $Y = JX$, $Z = X$ yields the first claim of the lemma.
It also follows that $R(\mathcal{Z},IX)IX = -R(\mathcal{Z},X)X$ and the same
is true for $J$ and $K$. Thus, 
$R(\mathcal{Z},X)X = R(\mathcal{Z},IJKX)IJKX = -R(\mathcal{Z},X)X$
which proves the second claim.
\end{proof}

Let us make a few remarks about the curvature of the Obata connection
on $\SU(3)$. Recall that we have the decomposition 
$\mathfrak{su}(3) = \frg_0 \oplus \frg_1$ where $\frg_0$ and
$\frg_1$ are one-dimensional $\bbH$-subspaces spanned
by $\EE$ and $W$ respectively. Note that it is possible to
choose $W$ in such a way that $\ob_W W \neq 0$: if not we would have
$\ob_W W = 0$ for all $W \in \frg_1$. Since the Obata connection is 
$\bbH$-linear this would also imply 
$\ob_{IW} W = \ob_{JW} W = \ob_{KW} W = 0$.
Since $\frg_1$ is one-dimensional we would have $\ob_X Y = 0$
and consequently $[X,Y] = 0$ for all $X$, $Y$ in $\frg_1$ which is
obviously not true. 

Recall that we have the following expression for the curvature:
$R(X,Y)Z = \mathrm{Alt}(\ob^2 Z)(X,Y)$, where 
$\ob^2 Z \in \exter^1G\otimes \exter^1G\otimes TG$
is a bilinear form with values in vector fields and $\mathrm{Alt}$ means
antisymmetrization of this form.
From the third part of Proposition \ref{prop_Euler} we obtain 
$R(X,Y)\EE = \mathrm{Alt}(\ob^2\EE)(X,Y) = 0$, thus $\frg_0$ lies in the
kernel of all the endomorphisms $R(X,Y)$.

We claim that $R(X,Y)\frg_1 = \frg_1$. Suppose that 
$X,Y \in \frg_0$, then the first Bianchi identity
implies $R(X,Y)\frg_1 = 0$. Next take $X \in \frg_0$ and $Y \in \frg_1$;
since the subspace $\frg_1$ is one-dimensional and
the curvature is $\SU(2)$-invariant,
it follows from the second part of Lemma \ref{lem_1} that $R(X,Y)Z = 0$
for any $Z \in \frg_1$.
Note that the Obata connection respects the grading on $\frg$, consequently
if $X,Y \in \frg_1$ then $R(X,Y)\frg_1 \subset \frg_1$. We remark that
the image of $R(X,Y)$ must be nontrivial for some $X,Y$, because otherwise
the Obata connection would be flat, which is not the case.
We will need the following statement.

\begin{prop}\label{prop_curv}
The holonomy group of the Obata connection contains an element
that acts identically on $\frg_0$ and multiplies $\frg_1$ by
a non-zero non-real quaternion.
\end{prop}
\begin{proof}
By Ambrose-Singer theorem (see e.g. \cite{Bes}) the Lie algebra
of the holonomy group contains all the endomorphisms $R(X,Y)$.
If $X, Y \in \frg_1$, then the endomorphism $R(X,Y)$ acts trivially on
$\frg_0$ and preserves $\frg_1$. Recall that $\frg_1$ is one-dimensional
over $\bbH$ and is generated by $W$. Put $Z_1 = R(W,IW)W$, 
$Z_2 = R(W,JW)W$, $Z_3 = R(W,KW)W$. 
It follows from the first part of Lemma \ref{lem_1} that the subspace
generated by $Z_1$, $Z_2$ and $Z_3$ is at least two-dimensional.
Indeed, otherwise we would have $Z_i = \alpha_i Z_0$, $i = 1,2,3$, for
some $\alpha_i \in \bbR$ and $Z_0 \in \frg_1$. Then by Lemma \ref{lem_1},
$(\alpha_1 + \alpha_3 J - \alpha_2 K)Z_0 = 0$, and this would imply $Z_i = 0$
meaning that the connection is flat, which is not true.
Thus the subalgebra generated by the endomorphisms $R(X,Y)$ with $X, Y\in \frg_1$
is at least two-dimensional. The claim of the proposition follows.
\end{proof}

The proof of the main theorem will be based on the following. 

\begin{prop}\label{prop_irred}
The holonomy of the Obata connection on $\SU(3)$ is irreducible. 
\end{prop}
\begin{proof}
The proof will consist of two parts. First, we will show that there
exist no left-invariant subbundles of $TG$ that are preserved by the holonomy.
Second, we will prove that there exist no holonomy-invariant subbundles at all.

Suppose that $\frh \subset \frg$ is a
subspace corresponding to a left-invariant subbundle preserved by the holonomy.
The left-invariance implies $\ob_X Y \in \frh$ for all $X \in \frg$ and $Y \in \frh$.
Let $V \in \frh$ and $V = V_0 + V_1$ where $V_0 \in \frg_0$, $V_1 \in \frg_1$.
Then 
$$
	\ob_\EE V = \ob_V \EE + [\EE, V] = V + [\EE, V_1],
$$
because $\EE$ lies in the center of $\frg_0$. We conclude that 
$[\EE, V_1] \in \frh$. Note that under identification (\ref{eqn_matr}) of $\frg$
with skew-Hermitian matrices, $\EE \in \frg$ corresponds to a diagonal
matrix with $D = -(b/2) Id$. It is easy to check that $(\mathrm{ad}_\EE)^2$ acts on $\frg_1$ by
real scalar multiplication, so we have $V_1 \in \frh$ and consequently $V_0 \in \frh$.
If there exists some $V \in \frh$ with $V_0 \neq 0$, it follows
that the Euler vector field $\EE$ lies in $\frh$ and this implies
$\frh = \frg$. Otherwise $\frh \subset \frg_1$. But this can happen
only if $\frh = 0$: it was remarked above that $\frg_1$ is $\bbH$-spanned
by $W$ with $\ob_W W \neq 0$ and it follows from $\bbH$-linearity of $\ob$
that $\ob_W V \neq 0$ and lies in $\frg_0$ for non-zero $V \in \frg_1$.

Now, we proceed to the second part of the proof.
Let $L_g\colon G \to G$ denote the left translation $h\mapsto gh$.
Suppose that there exists some (not left-invariant) proper subbundle $B$
preserved by the holonomy. Then for any $g\in G$ the subbundle $L^*_gB$
is also preserved by the holonomy, thus there exists a continuous family of
holonomy-invariant subbundles. We claim that it is possible to find a
holonomy-invariant subbundle $B$ with the following properties:
\begin{enumerate}
\item
	$\dim_\bbR B = 4$,
\item
	$B$ is invariant with respect to some of the complex structures,
\item
	$\dim_\bbR (B \cap L^*_g B)$ is either 0 or 4 for all $g\in G$.
\end{enumerate}
We will first find a subbundle that possesses the first two properties.
Consider holonomy-invariant subbundle $B$ of a minimal
possible dimension. Then $\dim_\bbR B$ must
be less or equal to 4, otherwise we could replace $B$ with $B \cap L^*_gB$
which is a proper subbundle of $B$ for some $g\in G$.
Next, we consider the four possibilities.
If $\dim_\bbR B = 1$, we can take the $\bbH$-span of $B$ and obtain
$\bbH$-invariant subbundle of real dimension 4.
If $\dim_\bbR B = 2$, we can take $B+IB$.
If $\dim_\bbR B = 3$, we can take $B+IB$ and obtain a subbundle of complex
dimension 2 or 3. In the former case we are done, and in the latter case,
we can intersect the subbundle with its left translation and decrease its dimension.
Consider the case when $\dim_\bbR B = 4$. Then $B$ is either $I$-invariant
or $B \cap IB = 0$. In the latter case $B\oplus IB$ is a complex
representation of the holonomy group. Suppose that it is irreducible.
Consider the operator $\mathcal{C}$ that fixes $B$ and multiplies $IB$ by $-1$.
This operator is {$I$-antilinear} and is preserved by the holonomy, and so is
the complex structure~$J$. The composition $J\mathcal{C}$ is
$I$-linear and thus by Schur's lemma must be equal to $\lambda Id$ with 
$\lambda \in \bbC$. But since $\mathcal{C}^2 = Id$, we have $\lambda \mathcal{C} = J$
and $-Id = J^2 = \lambda \mathcal{C} \lambda \mathcal{C} = |\lambda|^2 Id$ which
is impossible.
Consequently, the representation $B\oplus IB$ is reducible. We can replace $B$
with a proper $I$-invariant subbundle of $TG$ preserved by the holonomy
and of minimal dimension. The real dimension of $B$
must be less or equal to 4, otherwise we could replace $B$ with $B \cap L^*_gB$
for some $g\in G$. Since we are considering the case when
the minimal dimension of such a subbundle is greater or equal to 4,
$\dim_\bbR B = 4$ and we obtain a subbundle satisfying the first two
requirements. 

Now, if the subbundle $B$ does not possess the third property, then
there exists such $g\in G$ that $\dim_\bbR (B \cap L^*_g B) = 2$.
We can then replace $B$ with the $\bbH$-span of $B \cap L^*_g B$.
Since $B \cap L^*_g B$ is $I$-invariant, its $\bbH$-span will
have real dimension 4 and will satisfy all the three requirements.
 
Let the subbundle $B$ possess all the three properties listed above.
Since it can not be left-invariant, there exist $g_1, g_2 \in G$ such
that $B$, $L^*_{g_1} B$ and $L^*_{g_2} B$ form a triple of pairwise
complementary subbundles. Now we are going to use the following
observation.

\begin{lem}\label{lem_2}
Let $V$ be a $2n$-dimensional vector space, and $V_1$, $V_2$, $V_3$ three
pairwise complementary $n$-dimensional subspaces. Denote by $P_{ij}$ the
projection operator onto $V_i$ along $V_j$. Then the algebra generated by $P_{ij}$
is isomorphic to $\mathrm{Mat_2(\bbR)}$, the algebra of $2\times 2$ matrices.
\end{lem}
\begin{proof}
Consider the operator $A = P_{12}P_{31}$. It maps $V_2$ isomorphically onto $V_1$;
if we consider the decomposition $V=V_1\oplus V_2$ and identify $V_1$ and $V_2$
via $A$ then the operators $P_{12}$, $P_{21}$, $P_{12}P_{31}$ and $P_{21}P_{32}$ will
have the block matrix forms 
$\bigl(\begin{smallmatrix}
1 & 0\\0 & 0
\end{smallmatrix}\bigr)$,
$\bigl(\begin{smallmatrix}
0 & 0\\0 & 1
\end{smallmatrix}\bigr)$,
$\bigl(\begin{smallmatrix}
0 & 1\\0 & 0
\end{smallmatrix}\bigr)$,
$\bigl(\begin{smallmatrix}
0 & 0\\1 & 0
\end{smallmatrix}\bigr)$
respectively. 
\end{proof}

The holonomy group preserves a triple of pairwise complementary subbundles.
These subbundles are invariant with respect to some of the complex structures.
We will fix this complex structure and consider $TG$ as a complex vector bundle. 
The holonomy group must centralize the algebra generated by projections.
Therefore we can choose an isomorphism of vector spaces 
$\frg \simeq \bbC^2 \otimes_\bbC \bbC^2$ with the holonomy acting trivially on
the first factor and non-trivially on the second.
In particular, each operator in the holonomy group must have at most two
distinct eigenvalues. But Proposition \ref{prop_curv} implies that the holonomy
group contains an operator with three distinct eigenvalues 
(one real, equal to 1, and two complex-conjugate).
This contradiction ends the proof of irreducibility of the holonomy. 
\end{proof}

Now we are ready to prove the main theorem.

\begin{thm}
The holonomy group of the Obata connection on $\SU(3)$ with the
homogeneous hypercomplex structure is $\GL(2,\bbH)$.
\end{thm}
\begin{proof}
By Proposition \ref{prop_irred} the holonomy is irreducible.
The statement of the theorem follows from the
classification of irreducible holonomies from \cite{MS}. Indeed, 
the Obata connection on $\SU(3)$ does not preserve any metric (see Remark
{\ref{rem_euler}}). Here is the list of non-metric holonomy groups
with representation space $\bbR^8$ ($\grT{\bbF}$ denotes 
any connected Lie subgroup of $\bbF^*$): 

\begin{center}
\begin{tabular}{|c|c|}
\hline
From Table 2 in \cite{MS} & From Table 3 in \cite{MS} \\
\hline 
$\grT{\bbR}\cdot \SL(8, \bbR)$ & $\SL(2,\bbC)$ acting on $S^3\bbC^2$\\
$\grT{\bbC}\cdot \SL(4, \bbC)$ & $\bbC^*\cdot \SL(2,\bbC)$  acting on $S^3\bbC^2$\\
$\grT{\bbR}\cdot \SL(2, \bbH)$ & $\bbC^*\cdot \Sp(2,\bbC)$\\
$\Sp(4, \bbR)$ & $\SL(2,\bbR)\cdot \SO(p,q)$, $p+q = 4$\\
$\Sp(2, \bbC)$ & $\Sp(1)\cdot \SO(2,\bbH)$\\
$\bbR^*\cdot \SO(p,q)$, $p+q=8$ & \\
$\grT{\bbC}\cdot \SO(4, \bbC)$ & \\
$\grT{\bbR}\cdot \SL(m, \bbR)\cdot \SL(n, \bbR)$, $mn=8$ & \\
$\grT{\bbR}\cdot \SL(m, \bbH)\cdot \SL(n, \bbH)$, $mn=2$ & \\
\hline
\end{tabular}
\end{center}

The most of the entries in the list are obviously not contained
in $\GL(2,\bbH)$ because of dimension reasons or because they
do not preserve any complex structure. Note that the action
of $\SL(2,\bbC)$ on $S^3\bbC^2$ does not preserve quaternionic
structure, because it does not commute with any non-scalar $\bbR$-linear
operator.
Indeed, let $A\in \emrp(\bbR^8)$ be a real endomorphism commuting with the
action of $SL(2,\bbC)$ on $S^3\bbC^2 \simeq \bbR^8$. Consider the weight
decomposition $S^3\bbC^2 = \bigoplus_{\lambda} V_\lambda$ where $V_\lambda$ is
an eigenspace of 
$H = \bigl(\begin{smallmatrix}
1 & 0\\0 & -1
\end{smallmatrix}\bigr) \in \frsl(2,\bbC)$ with eigenvalue $\lambda$,
and the sum runs over $\lambda = 3$, $1$, $-1$, $-3$. Since the
eigenvalues are real, $A$ must preserve the eigenspaces $V_\lambda$.
Moreover, $A$ has to be $\bbC$-linear because it commutes with 
$\ii H \in \frsl(2,\bbC)$ which has the same eigenspaces $V_\lambda$ with 
eigenvalues $\ii \lambda$.
By Schur's lemma $A$ must be equal to a scalar operator.
Thus, the groups $\SL(2,\bbC)$ and $\bbC^*\cdot \SL(2,\bbC)$
can not occur as holonomy groups of the Obata
connection.

The list contains only one proper subgroup of $\GL(2,\bbH)$, namely
$\SL(2,\bbH)$. But if the holonomy was $\SL(2, \bbH)$, the Obata connection
would preserve a holomorphic volume form, and this is impossible
(see Remark {\ref{rem_euler}}). Thus, the holonomy group must coincide
with $\GL(2,\bbH)$.
\end{proof}

\hfill

\hfill

\small{\noindent{\sc Andrey Soldatenkov \\
Laboratory of Algebraic Geometry, HSE, \\ 
7 Vavilova Str., Moscow, Russia, 117312.}}

\end{document}